\documentclass[12pt]{article}
\usepackage[english,activeacute]{babel}
\usepackage{amsmath,csquotes,amsfonts,amssymb,amstext,amsthm,amscd,mathrsfs,amsbsy,graphics,graphicx}
\usepackage{xypic}
\usepackage[all]{xy}
\newtheorem{theo}{Theorem}[section]
\newtheorem{pro}[theo]{Proposition}
\newtheorem{coro}[theo]{Corollary}
\newtheorem{lem}[theo]{Lemma}

\theoremstyle{definition}
\newtheorem{defi}[theo]{Definition}
\newtheorem{exam}[theo]{Example}
\newtheorem{rem}[theo]{Remark}

\begin{document}
\title{The Deligne pairing and a functorial Riemann Roch theorem in positive characteristic}
\author{\\Quan Xu \footnote {Email to: qxu@math.tsinghua.edu.cn.  \ Supported by post-doctorate fellowship in Yau Mathematical Sciences Center, Tsinghua University.} 
}
\maketitle

\begin{abstract}
In this paper, we prove the functorial Riemann-Roch theorem in positive characteristic for a smooth and projective morphism with any relative dimension. In the case of relative dimension $1$, we have given an analogue with Deligne's functorial Riemann Roch theorem in previous author's paper. For any relative dimension, our result can deduce an analogue to the Knudsen-Mumford extension. The present result is a generalization, which mainly originated from the extended Deligne pairing by S. Zhang and the Adams Riemann Roch theorem in positive characteristic by R. Pink and D. R\"{o}ssler.
\end{abstract}

\section{Introduction}
For a scheme $X$, one has Chern character map $\text{ch}: K_{0}(X)\to A(X)_{\mathbb{Q}}$ from (Grothendieck ) K-group to (the localized ) Chow Group
of $X$.
Furthermore, given a morphism of schemes $f:X\to Y$, we have  $f_{*}:  K_{0}(X)\to  K_{0}(Y)$ from K-group of $X$ to K-group of $Y$ and
$f_{!}: A(X)\to A(Y)$ from Chow group of $X$ to Chow group of $Y$.
 An natural question is to ask : whether Chern character map $\text{ch}$ can have a commutativity with the maps $f_{*}, f_{!}$ for the given morphism $f$ , that is, $\text{ch}(f_{*}(E))=f_{!}(\text{ch}(E))$ for 
 any $E\in K_{0}(X)$ ? These concepts including Grothendieck $K$-group, Chern character and Chow group, $f_{*}, f_{!}$ can be found in 
 lots of reference, for instance, \cite{Ful}. On one hand,  Grothendieck-Riemann-Roch theorem(GRR, in abrreviation) answers the question: the commutativity 
 does not hold, but the difference is $\text{Todd}$ classes.\\
(GRR, \cite{Berth}) Let $f:X\to Y$ be a proper morphism of regular schemes. Then for  any $E\in K_{0}(X)$, the equality 
 $$\text{ch}(f_{*}E)\cdot \text{td}(T_{Y})=f_{!}(\text{ch}(E)\cdot\text{td}(T_{X}))$$ holds
 in $A(Y)_{\mathbb{Q}}$.
 By taking the degree one under the both sides of the equality in GRR,  one has a special case of GRR ( \cite{FL}):
$$\text{c}_{1}(Rf_{*}E)=f_{!}(\text{ch}(E)Td(\mathrm{T}_{f}))_{(1)}$$

 In order to clarify the functoriality of Riemann-Roch theorem, we need to have a close look to the Grothendieck K-group and the Chow group.
Roughly speaking, the $K$-group is a free abelian group generated by isomorphic classes of all coherent sheaves (or locally free and coherent sheaves) over a scheme modulo out an equivalent relation (the class of the exact sequence of coherent sheaves). The  Chow group $A(X)=\oplus A^{i}(X)$ of a scheme $X$ is a direct sum of $A^{i}(X)$ which is the free abelian group generated by all closed sub-schemes of co-dimension $i$ modulo out linear equivalence (see \cite{Ful}). Therefore, we can not obtain the information of equivalence relations from the equality in GRR theorem. 
 
 A question proposed by P. Deligne is whether or not  we can have an isomorphism for a given morphism of schemes and from the isomorphism, can extract more information (such as equivalence relations) besides the equality in  GRR theorem? To be precise, exists there a categorical refinement of  GRR theorem? In the categorical refinement of GRR theorem, the Grothendieck $K$-group and Chow group can have categorical replacement such that there is a functor between two categories and the equality in GRR theorem can be refined by functorial isomorphism in a suitable category. The kind of question is just referred to the functorial Riemann-Roch theorem. 
 
 From the special case of GRR above, P.Deligne gave a certain answer which is Deligne's functorial Riemann-Roch theorem:
 
 (\cite{Del}, Theorem 9.9): Let $f:C\rightarrow S$ be a smooth family of proper curves and $L$ be a line bundle over $C$. There exists a unique, up to a sign, functorial isomorphism of line bundles
\begin{align}
(\det &Rf_{*}L)^{\otimes 18 }\notag\\
&\cong (\det  Rf_{*}\mathcal{O})^{\otimes 18 }\otimes (\det  Rf_{*}(L^{\otimes 2}\otimes 
\omega^{-1}))^{\otimes 6}\otimes  (\det  Rf_{*}(L\otimes \omega^{-1}))^{\otimes (-6)}.\notag
\end{align}
in the Picard category of $S$, where these notations are explained as follows

(1a) $\det$ is the extended determinant functor (see Sect. 2.2);

(1b) $Rf_{*}L$ is a complex which consists of all $R^{i}f_{*}L$ with $i\geqslant 0$;

(1c) $\omega$ is a canonical sheaf of the morphism $f$.\\
In fact, in Deligne's theorem, $K_{0}$-group is replaced by the virtual category (see Def.\ref{vircon}), but there is no need of a categorical replacement of Chow group. The functor that we are searching for is the determinant functor (see Sect. 2.2).  The isomorphism in Deligne's functorial Riemann-Roch theorem was also suggested by the special case of GRR above which says that the classes of line bundles
are the same in the Picard group ( \cite{Thom}). Observe that Deligne's functorial Riemann Roch theorem depends on the Mumford isomorphism. Therefore, we recall the Mumford isomorphism:

(Mumford isomorphism) Let $f:C\rightarrow S$ be a flat local complete intersection generically smooth proper morphism with geometrically connected fibers of dimension $1$, with $S$ any connected normal Noetherian locally factorial scheme.  Then for the canonical sheaf $\omega$ of the morphism $f$, we have a canonical isomorphism 
  $\det Rf_{*}(\omega^{n})\cong(\det Rf_{*}(\omega))^{6n^{2}-6n+1}$ for any integer $n\geqslant 0$.

On another hand, in \cite{Berth}, besides Grothendieck Riemann-Roch theorem, by introducing the Adams operation, Grothendieck also proved the Adams-Riemann-Roch theorem: 
let $f:X\rightarrow Y$ be projective local complete intersection morphism of  dimension $r$ and $Y$ a quasi-compact scheme with an ample invertible sheaf,
then the equality $$\psi^{k}(R^{\bullet}f_{*}(E))=R^{\bullet}f_{*}(\theta ^{k}(\Omega_{f})^{-1}\otimes\psi^{k}(E))$$ for 
 any integer $k\geq 2$ holds in $K_{0}(Y)\otimes \mathbb{Z}[\frac{1}{k}]$. The notations are given as follows:

(a) The symbol $\psi^{k}$ is the $k$-th Adams operation and $\theta ^{k}$ is the $k$-th Bott class operation (see Sect. 3.2 and 3.3);

(b) For a vector bundle $E$, $R^{\bullet}f_{*}(E)=\sum_{i\geq 0}(-1)^{i}R^{i}f_{*}(E)$ where $R^{i}f_{*}$ is the higher direct image functor of the push-forward $f_{*}$ (see Sect. 2.1);

(c ) For a quasi-compact scheme $Y$, $K_{0}(Y)$ is the Grothendieck group of locally free coherent sheaves of $\mathcal{O}_{Y}$-modules (see Sect. 2.1);

(d) The symbol $\Omega_{f}$ is the relative differentials of the morphism $f$. When $f$ is a smooth and projective morphism, $\Omega_{f}$ is locally free sheaf (see \cite{Hart}).

Under a suitable condition (for instance, $f$ is a projective and smooth morphism), the GRR theorem and the Adams-Riemann-Roch theorem are equivalent (see \cite{FL}).  The equivalence of two theorems partially explains why the Chow group (or a categorical replacement of Chow group) does not appear in Deligne's functorial Riemann-Roch theorem. Originally, there is no restriction to the characteristic of schemes in the course of statement and proof of theorems. For a projective and local complete intersection morphism (see \cite{FL}, Pag.86), a general strategy of proving these theorems is to verify the theorem for the closed immersion and smooth projection, respectively. And then the theorems is valid in their composition of the closed immersion and smooth projection. Usually, the technique of the deformation to the normal cone (see \cite{Man} or \cite{FL}) will be used, which is not an easy work.  

In \cite{Pi}, in the setting of the positive characteristic, R. Pink and D. R\"{o}ssler proved the following version of the Adams-Riemann-Roch theorem:
\\Let $f:X\rightarrow Y$ be projective and smooth of relative dimension $r$, where
$Y$ is a quasi-compact scheme of characteristic $p>0$ and carries an ample invertible sheaf. Then the following equality
$$\psi^{p}(R^{\bullet}f_{*}(E))=R^{\bullet}f_{*}(\theta ^{p}(\Omega_{f})^{-1}\otimes\psi^{p}(E))~~~~~(*)$$
holds in $K_{0}(Y)[\frac{1}{p}]:=K_{0}(Y)\otimes_{\mathbb{Z}}\mathbb{Z}[\frac{1}{p}].$

In their proof, Pink and R\"{o}ssler did not use decomposition and composition of projective morphism, and the deformation to the normal cone, either. Instead,  by considering the relative Frobenius morphism and absolute Frobenius morphism,  they prove a version of Adams-Riemann-Roch theorem by further construction of the Bott class, which completely embodies the advantage of the positive characteristic.

In this paper, by Deligne's functorial Riemann-Roch theorem and Pink and R\"{o}ssler's Adams-Riemann-Roch theorem in positive characteristic, we will prove a  functorial version of Adams-Riemann-Roch theorem in positive characteristic (see Thm\ref{mrs}) for a projective and smooth morphism of any relative dimension, which is our main result. For any relative dimension, our theorem also deduces an analogue with the Knudsen-Mumford extension in positive characteristic (see Cor. \ref{kme}).
 
The structure of the paper is  as follows. In the section 2, we provide the necessary materials for other sections, including Grothendieck $K$-groups, the Picard category, virtual category in section 2.1. In section 2.2 , according to the original definition of the determinant functor which is given Knudsen and Mumford  \cite{Kund}, we aim at defining specifically the determinant functor, from the exact category to the Picard category, which factorizes through the virtual category (Def. \ref{vdk}) . The definition is essentially important, which makes our theorem functorial. In section 2.3, we will introduce the extended Deligne pairing Defined by S. Zhang. Then the important property of the Deligne pairing is Prop. \ref{d,p} which is an analogue with original Deligne pairing. Its direct corollary (see Cor. \ref{trs}) will be used in the proof of our theorem.
 
In Section 3, the Adams-Riemann-Roch theorem will be stated including all  necessary ingredients. Moreover, we are more concerned with its proof in the context of positive characteristic since the ideas in the proof will be employed in our results.  
 
In Section 4, our main results are shown, including a theorem and two corollaries. Theorem \ref{mrs} can be viewed as the functorial Adams-Riemann-Roch in positive characteristic for a projective and smooth morphism of any relative dimension. In the case of relative dimension $1$, Cor.\ref{functor} completely agrees with Deligne's functorial Riemann-Roch theorem and Mumford isomorphism when the characteristic is 2. As another application, Cor.\ref{kme} is also an analogue with the Knudsen-Mumford extension in positive characteristic.
 \section{Preliminaries}
From now on,  whenever a scheme is mentioned in the paper,  it means that it is always a quasi-compact scheme, unless
we use a different assumption for schemes.
\subsection{ Grothendieck groups and the virtual category}
Let $X$ be a scheme. the first ingredient in Riemann-Roch theorem is Grothendieck $K$-groups.
Usually, $K_{0}(X)$ and $K_{0}^{'}(X)$ are denoted the Grothendieck groups of locally free shaves and coherent sheaves on a scheme $X$, respectively.
The following are basic facts about Grothendieck groups:

(1) The tensor product of $\mathcal{O}_{X}$-modules makes the group $K_{0}(X)$ into a commutative unitary ring and the inverse image of locally free sheaves under
any morphism of schemes $X^{'}\rightarrow X$ induces a morphism of unitary rings $K_{0}(X)\rightarrow K_{0}(X^{'})$ (see \cite{Man}, \S 1);

(2) The obvious group morphism $K_{0}(X)\rightarrow K_{0}^{'}(X)$ is an isomorphism if $X$ is regular and carries an ample invertible sheaf (see \cite{Man}, Thm. 1.9);

(3) Let $f:X\rightarrow Y$ be a projective local complete intersection morphism of schemes (A morphism $f 
: X\rightarrow Y$ is called a  local complete intersection morphism if $f $  is a composition of morphisms as $X\rightarrow \text{P} \rightarrow Y$ where the first morphism is a regular embedding and the second is a smooth morphism. See \cite{Ful}  or \cite{FL}) and $Y$ carries an ample invertible sheaf. There is
a unique group morphism $\text{R}^{\bullet}f_{*}:K_{0}(X)\rightarrow K_{0}(Y)$ which sends the class of a locally free coherent sheaf $E$ on $X$ to the class of the strictly perfect complex (The strictly perfect complex will be defined in Sect. 2.2) $\text{R}^{\bullet}f_{*}E$ in $K_{0}(Y)$, where $\text{R}^{\bullet}f_{*}E$ is defined 
to be $\sum_{i\geq 0}(-1)^{i}\text{R}^{i}f_{*}E$ and $\text{R}^{i}f_{*}E$ is an element in $K_{0}^{'}(Y)$ and is viewed as an element in $K_{0}(Y)$ under $K_{0}^{'}(Y)\cong K_{0}(Y)$ in the sense of (2) above (see \cite{Berth}, IV, 2.12). 

In \cite{Del}, Deligne defined a categorical refinement of the Grothendieck groups, which are referred to  the additive category and the exact category and the abelian category. These definitions can be found in a lot references, for instance, in \cite{Quill}. In order to give the definition of the virtual category, we shall need the definition of a groupoid, especially of a specific groupoid called the Picard groupoid (see \cite{Del}, \S 4).

Picard category is a groupoid $P$ (a non-empty category in which all morphisms are isomorphisms  ) with a  functor  $\oplus: P\times P\rightarrow P$ and an associativity constraint 
for the functor $\oplus$  and an commutative constraint such that two kinds of constraints are compatible. Besides, for any object $Y$ in $P$, these functors $X\rightarrow X\oplus Y$ and $X\rightarrow Y\oplus X$ are auto-equivalence of $P$. (See \cite{Jsm} for two kinds of constraint and their compatibility, also \S 4 in \cite{Del} for precise definition). The Picard category we describe is essentially a commutative Picard category in \S 4,
\cite{Del}.
\begin{exam}\label{lbd}
Let $X$ be a scheme. We denote by $\mathscr{P}_{X}$ the category of graded invertible 
$\mathcal{O}_{X}$-modules. An object of $\mathscr{P}_{X}$ is a pair $(L,\alpha)$ where $L$ is an invertible $\mathcal{O}_{X}$-module and $\alpha$
is a continuous function: $$\alpha: X\rightarrow \mathbb{Z}.$$ 

A homomorphism $h:(L,\alpha)\rightarrow (M,\beta)$ is a homomorphism $L\rightarrow M$ of $\mathcal{O}_{X}$-modules such that for each $x\in X$  we have:
$$\alpha(x)\neq\beta(x)\Rightarrow h_{x}=0.$$

We denote by $\mathscr{P}is_{X}$ the subcategory of $\mathscr{P}_{X}$ whose morphisms are all isomorphism.
The tensor product of two objects in $\mathscr{P}_{X}$ is given by: $$(L,\alpha)\otimes(M,\beta)=(L\otimes M, \alpha+\beta).$$
For each pair of objects $(L,\alpha), (M,\beta)$ in $\mathscr{P}_{X}$  we have an isomorphism:
$$\xymatrix{\psi_{(L,\alpha), (M,\beta)}: (L,\alpha)\otimes(M,\beta)\ar[r]^(.60){\sim}&(M,\beta)\otimes(L,\alpha)  }$$
defined as follows: If $l\in L_{x}$ and $m\in M_{x}$ then 
$$\psi(l\otimes m)=(-1)^{\alpha(x)+\beta(y)}\cdot m\otimes l.$$\\
Clearly: $$\psi_{(M,\beta),(L,\alpha)}\cdot\psi_{(L,\alpha), (M,\beta)}=1_{(L,\alpha)\otimes(M,\beta)}$$\\
We denote by $1$ the object $(\mathcal{O}_{X},0)$. A right inverse of an object $(L,\alpha)$ in $\mathscr{P}_{X}$ will be an object $(L^{'},\alpha^{'})$
together with an isomorphism $$\xymatrix{\delta: (L,\alpha)\otimes(L^{'},\alpha^{'})\ar[r]^(.70){\sim}& 1}$$\\
Of course $\alpha^{'}=-\alpha$.
A right inverse will be considered as a left inverse via:
$$\xymatrix{\delta: (L^{'},\alpha^{'})\otimes(L,\alpha)\ar[r]_(.55){\sim}^(.55){\psi}& (L,\alpha)\otimes(L^{'},\alpha^{'})\ar[r]_(.70){\sim}^(.70){\delta}&1 }.$$
According to the definition of the Picard category, further verification implies that $\mathscr{P}is_{X}$ is a Picard category.
\end{exam}
Now, we can give the definition of Deligne's virtual category. 
By an admissible filtration in an exact category we mean a finite sequence of admissible monomorphisms $0=A^{0}\rightarrowtail A^{1}\rightarrowtail
\cdots \rightarrowtail A^{n} = C.$
\begin{defi}\label{vircon}(see \cite{Del}, Pag. 115)
The virtual category $V(\mathcal{C})$ of an exact category $\mathcal{C}$ is
a Picard category, together with a functor $\{~\} : (\mathcal{C}, iso)\rightarrow V(\mathcal{C})$ (Here, the
first category is the subcategory of $\mathcal{C}$ consisting of the same objects and the
morphisms are the isomorphisms of $\mathcal{C}$.), with the following universal property:\\
Suppose we have a functor $[~] : (\mathcal{C}, iso)\rightarrow P$ where $P$ is a Picard category,
equipped with (a) and (b) satisfying the axioms © and (d) below.

(a) Additivity on exact sequences, i.e., for an exact sequence $A \rightarrow B \rightarrow C $ ($A\rightarrow B$ is a admissible monomorphism and $B\rightarrow C$ is a admissible epimorphism), 
we have an isomorphism $[B]\cong[A]\oplus [C]$, functorial with respect to isomorphisms of exact sequences.

(b) A zero-object of $C$ is isomorphically mapped to a zero-object in $P$ (According to the definition of  the Picard category, it implies the existence of the unit object which is also called zero-object. See \cite{Del}, \S 4.1.).

(c ) The functor [ ]  is compatible with admissible filtrations, is compatible with admissible filtrations, i.e., for an admissible filtration $C\supset B\supset A\supset 0$,
the diagram of isomorphisms from (a) 
$$\xymatrix{[C]\ar[rr]\ar[d]&    & [A]\oplus [C/A]\ar[d]\\
    [B]\oplus[C/B]\ar[rr] & & [A]\oplus[B/A]\oplus[C/B]                  }$$
 is commutative.

(d) If $f: A\rightarrow B$ is an isomorphism and $\sum$ is the exact sequence $0\rightarrow A\rightarrow B$ (resp. $A\rightarrow B\rightarrow 0$ ), then $[f]$ (resp. $[f]^{-1}$) 
is the composition
$$\xymatrix{[A]\ar[r]_(.35){\sum}&[0]\oplus[B]\ar[r]_(.6){(b)}&[B]   }$$
$$(resp. \xymatrix{[B]\ar[r]_(.35){\sum}&[A]\oplus[0]\ar[r]_(.6){(b)}&[A]})$$
where (b) in the diagram above means that the morphism is from (b).\\
Then the conclusion is that the functor $[~] :(\mathcal{C}, iso) \rightarrow P$ factors uniquely up to
unique isomorphism through $(\mathcal{C}, iso)\rightarrow V(\mathcal{C})$. 
\end{defi}
Roughly speaking, for an exact category $\mathcal{C}$, $V(\mathcal{C})$ is a universal Picard category with a functor $[~]$ satisfying given properties. In practice, the functor $[~]$ usually can be chosen as the determinant functor we will 
define in the next subsection.
The definition on virtual category above is an abstractly algebraic description. In \cite{Del}, Deligne also provided a topological definition for the virtual category  of a small exact category.
\begin{theo}\label{v,k}
Let $\mathcal{C}$ be a exact category and $\mathcal{B}Q\mathcal{C}$ is the geometrical realization of the Quillen Q-construction of $\mathcal{C}$ (See \S 1, \cite{Quill} ). Objects are loops in $\mathcal{B}Q\mathcal{C}$ around a fixed zero-point wich is a zero object in $\mathcal{C}$ , and morphisms are homotopy-classes of homotopies
of loops.The addition is the usual addition of loops. Then $\mathcal{B}Q\mathcal{C}$ is the virtual category $V(\mathcal{C})$of $\mathcal{C}$.
\end{theo}
\begin{proof}
See \cite{Del}, Pag. 114.
\end{proof}

\subsection{The determinant functor}\label{sdf}

In this subsection, we will consider the determinant functor and mainly consult \cite{Kund}. In loc. cit., the determinant functor can be defined in several backgrounds.
But the case we are most interested in  is the determinant functor from some subcategory of derived category to the subcategory of the category $\mathscr{P}_{X}$ of graded line bundles. 

In the following,  we denote by $\mathscr{C}_{X}$ the category  of locally free of finite type $\mathcal{O}_{X}$-modules for a scheme $X$.
 
\begin{defi}
If $E\in \text{ob}(\mathscr{C}_{X})$, we define: $\det ^{*}(F)=(\wedge^{max}F, \text{rank}F)$ \\ (where $(\wedge^{max}F)_{x}=\wedge^{\text{rank}F_{x}}F_{x}$).
\end{defi}

For every short exact sequence of objects in $\mathscr{C}_{X}$ 
$$\xymatrix{0\ar[r]&F^{'}\ar[r]^{\alpha}& F\ar[r]^{\beta}&F^{''}\ar[r]&0                 }$$ we have an isomorphism,
$$\xymatrix{i^{*}(\alpha,\beta): \det^{*}F^{'}\otimes \det^{*}F^{''}\ar[r]^(.70){\sim}&\det^{*} F   },$$ such that locally,
$$i^{*}(\alpha,\beta)((e_{1}\wedge\ldots\wedge e_{l})\otimes(\beta f_{1}\wedge\ldots\beta f_{s}))=\alpha e_{1}\wedge\ldots\alpha e_{l}\wedge f_{1}\wedge\ldots f_{s}$$
for $e_{i}\in\Gamma (U,F^{'})$ and $f_{j}\in\Gamma (U,F^{''})$.

\begin{defi}
 If $F^{i}$ is an indexed object of $\mathscr{C}_{X}$ we define:
\[
\det(F^{i})=
\begin{cases}
\det^{*}(F^{i}) &\text{if $i$ even};\\
 \det^{*}(F^{i})^{-1}    &\text{if $i$ odd}.

\end{cases}
\]
If $$\xymatrix{ 0\ar[r]& F^{i^{'}}\ar[r]^{\alpha^{i}}&F^{i}\ar[r]^{\beta^{i}}&F^{i^{''}}\ar[r]&0} $$
is an indexed short exact sequence of objects in $\mathscr{C}_{X}$, we define
\[
i(\alpha^{i},\beta^{i})=
\begin{cases}
i^{*}(\alpha^{i},\beta^{i})&\text{if $i$ even};\\
 i^{*}(\alpha^{i},\beta^{i})^{-1}   &\text{if $i$ odd}.

\end{cases}
\]
Usually, for a object $F$ in $\mathscr{C}_{X}$, we view the object as the indexed object by $0$, i.e., $\det(F)=\det^{*}(F)$.

\end{defi}
We also denote by $\mathscr{C}^{\cdot}_{X}$ the category of the bounded complex of objects in $\mathscr{C}_{X}$ over a scheme $X$.
\begin{defi}
If $F^{\cdot}$ is an object of $\mathscr{C}^{\cdot}_{X}$, we define $$\det(F^{\cdot})=\cdots\otimes\det(F^{i+1})\otimes\det(F^{i})\otimes\det(F^{i-1})\otimes\cdots$$
Furthermore, if $$\xymatrix{ 0\ar[r]& F^{\cdot'}\ar[r]^{\alpha}&F^{\cdot}\ar[r]^{\beta}&F^{\cdot''}\ar[r]&0} $$
is a short exact sequence of objects in $\mathscr{C}^{\cdot}_{X}$ we define 
$$\xymatrix{i(\alpha,\beta):\det(F^{\cdot'})\otimes \det(F^{\cdot''})\ar[r]^(.70){\sim}& \det(F^{\cdot})  }$$
to be the composite:
$$\det(F^{\cdot'})\otimes \det(F^{\cdot''})=\cdots\otimes\det(F^{i'})\otimes\det(F^{i-1'})\otimes\cdots$$
$$\xymatrix{\otimes\det(F^{i''})\otimes\det(F^{i-1''})\otimes\cdots\ar[r]^(.55){\sim}&\cdots\otimes\det(F^{i'})\otimes\det(F^{i''}) }$$
$$\xymatrix @C=0.7in{\otimes\det(F^{i-1'})\otimes\det(F^{i-1''})\otimes\cdots\ar[r]^(.65){\otimes_{i}i(\alpha^{i},\beta^{i})}_(.65){\sim}&\cdots\otimes\det(F^{i})  }$$
$$\otimes\det(F^{i-1})\otimes\cdots=\det(F^{\cdot}) .$$
\end{defi}
In \cite{Kund}, it is proved that there is one and, up to canonical isomorphism, only one determinant $(f,i)$ from $\mathscr{C}is_{X}$ (resp. $\mathscr{C}^{\cdot}is_{X}$) to $\mathscr{P}is_{X}$, which we write $(\det, i)$, 
where $\mathscr{C}is_{X}$ (resp. $\mathscr{C}^{\cdot}is_{X}$) is the category with same objects from $\mathscr{C}_{X}$ (resp. $\mathscr{C}^{\cdot}_{X}$) and the morphisms being
all isomorphisms (resp. quasi-isomorphisms). In case of repeating, we do not give the definitions of the determinant functor from from $\mathscr{C}is_{X}$ (resp. $\mathscr{C}^{\cdot}is_{X}$) to $\mathscr{P}is_{X}$, because the definitions are completely similar
to the following definition of the extended functor. For the precise definitions and proofs, see \cite{Kund}, Pag. 21-30.

In order to extend the determinant functor to the derived category in \cite{Kund}, we need to recall the definitions about the perfect complex and the strictly perfect complex:

\begin{defi}\label{exd} In \cite{Berth}, a perfect complex $\mathcal{F}^{\cdot}$ on a scheme $X$ means a complex of $\mathcal{O}_{X}-$modules (not necessarily quasi-coherent) such
that locally on $X$ there exists a bounded complex $\mathcal{G}^{\cdot}$ of free $\mathcal{O}_{X}-$modules of finite type and a quasi-isomorphism:
$$\mathcal{G}^{\cdot}\rightarrow\mathcal{F}^{\cdot}\mid_{U}$$ for any open subset $U$ of a covering of $X$. 

A strictly perfect complex $\mathcal{F}^{\cdot}$ on a scheme $X$ 
is a bounded complex of locally free $\mathcal{O}_{X}-$modules of finite type.
\end{defi}

 In other words, a perfect complex is locally quasi-isomorphic to a strictly perfect complex.
We denote by $\text{Parf}_{X}$ the full subcategory of $\text{D(Mod}X)$ (the derived category of bounded complexes of  sheaves of $\mathcal{O}_{X}-$modules) whose objects are perfect complexes and denote by $\text{Parf-is}_{X}$ the subcategory of
$\text{D(Mod}X)$  whose objects are perfect complexes and morphisms are only quasi-isomorphisms.
\begin{defi}\label{ex,de}(see \cite{Kund}, Pag. 40)
An extended determinant functor $(f, i)$ from  $\text{Parf-is}$ to $\mathscr{P}is$ consist of the following data:

I) For every scheme $X$, a functor
$$f_{X}: \text{Parf-is}_{X}\rightarrow\mathscr{P}is_{X}$$ such that $f_{X}(0)=1$.

II) For every short exact sequence of complexes $$ \xymatrix{ 0\ar[r]& F\ar[r]^{\alpha}&G\ar[r]^{\beta}&H\ar[r]&0} $$ in $\text{Parf-is}_{X}$,
 we have an isomorphism:
$$i_{X}(\alpha,\beta):\xymatrix{f_{X}(F)\otimes f_{X}(H)\ar[r]^(.65){\sim}&f_{X}(G)}$$ such that for the particular short exact sequences
$$\xymatrix {0\ar[r]&H\ar@{=}[r]&H\ar[r]&0\ar[r]&0}$$
and 
$$\xymatrix {0\ar[r]&0\ar[r]&H\ar@{=}[r]&H\ar[r]&0}$$
we have : $i_{X}(1,0)=i_{X}(0,1)=1_{f_{X}(H)}$.\\
We require that:

i) Given an isomorphism of short exact sequences of complexes
$$\xymatrix{0\ar[r]&F\ar[r]^{\alpha}\ar[d]^{u}&G\ar[r]^{\beta}\ar[d]^{v}&H\ar[r]\ar[d]^{w}&0 \\
0\ar[r]&F^{'}\ar[r]^{\alpha ^{'}}&G^{'}\ar[r]^{\beta ^{'}}&H^{'}\ar[r]&0  }$$
the diagram 
$$\xymatrix{f_{X}(F)\otimes f_{X}(H)\ar[r]^(.65){i_{X}(\alpha,\beta)}_(.65){\sim}\ar[d]^{f(u)\otimes f_{X}(w)}_{\wr}&f(G)\ar[d]^{f_{X}(v)}_{\wr}\\
f(F^{'})\otimes f(H^{'})\ar[r]_(.65){i_{X}(\alpha^{'},\beta^{'})}^(.65){\sim} &  f(G^{'})      }$$
commutes.

ii) Given a exact sequence of short exact sequences of complexes, i.e., a commutative diagram
$$\xymatrix{    &   0\ar[d]&0\ar[d]&0  \ar[d]\\
0\ar[r]&F\ar[r]^{\alpha}\ar[d]^{u}&G \ar[r]^{\beta}\ar[d]^{u^{'}}&H\ar[r]\ar[d]^{u^{''}}&0 \\
0\ar[r]&F^{'}\ar[r]^{\alpha^{'}}\ar[d]^{v}&G^{'}\ar[r]^{\beta^{'}}\ar[d]^{v^{'}}&H^{'}\ar[r]\ar[d]^{v ^{''}}&0 \\
0\ar[r]&F^{''}\ar[r]^{\alpha^{''}}\ar[d]&G^{''} \ar[r]^{\beta^{''}}\ar[d]&H^{''}\ar[r]\ar[d]&0 \\
            &   0  &0  &0       }$$
the diagram:
$$\xymatrix{f_{X}(F)\otimes f_{X}(H)\otimes f_{X}(F^{\cdot''})\otimes f_{X}(H^{''})\ar[rrr]^(.62){i_{X}(\alpha,\beta)\otimes i_{X}(\alpha^{''},\beta^{''})}_(.62){\sim}\ar[d]^{i_{X}(u,v)\otimes i_{X}(u^{''},v^{''})}_{\wr}&  & &f_{X}(F^{\cdot})\otimes f_{X}(H^{\cdot})\ar[d]^{i_{X}(u^{'}, v^{'})}_{\wr}\\
f_{X}(F^{'})\otimes f_{X}(H^{'})\ar[rrr]_(.65){i_{X}(\alpha^{'},\beta^{'})}^(.65){\sim} & &  &f_{X}(G^{'})      }$$
commutes.

iii) $f$ and $i$ commutes with base change. More precisely, this means: \\
For every morphism of schemes 
$$g: X\rightarrow Y$$
we have an isomorphism 
$$\eta(g):  \xymatrix{f_{X}\cdot \text{L}g^{*}\ar[r]^{\sim}&g^{*}f_{Y}}$$ such that for every short exact sequence of complexes
$$\xymatrix{0\ar[r]&F^{\cdot}\ar[r]^{u}&G^{\cdot}\ar[r]^{v}&H^{\cdot}\ar[r]&0}$$
the diagram:
$$\xymatrix{f_{X}(\text{L}g^{*}F^{\cdot})\otimes f_{X}(\text{L}g^{*}H^{\cdot})\ar[d]^{\eta \cdot\eta}_{\wr}\ar[rr]^(.60){i_{Y}(\text{L}g^{*}(u,v))}_(.60){\sim}&& f_{X}(\text{L}g^{*}G^{\cdot})\ar[d]^{\eta}_{\wr}\\
g^{*}f_{Y}(F^{\cdot})\otimes g^{*}f_{Y}(H^{\cdot})\ar[rr]^{i_{Y}(u,v)}_{\sim}& &g^{*}f_{Y}(F^{\cdot})
}$$
commutes, where $\text{L}g^{*}$ is the left derived functor of the morphism $g$ and exists for the category whose objects are short exact sequences of 
complexes of three objects in $\text{Mod}(Y)$ and whose morphisms are triples
such that the resulting diagram (like the diagram in i) but not isomorphism in general) commutes (see \cite{Kund}, Prop. 3). Moreover if 
$$\xymatrix{X\ar[r]^{g}&Y\ar[r]^{h}&Z}$$ are two consecutive morphisms, the diagram:

$$\xymatrix{ f_{X}(\text{L}g^{*}\text{L}h^{*})\ar[r]^{\eta(g)}_{\sim}\ar[d]^{f_{X}(\theta)}_{\wr}& g^{*}f_{Y}\text{L}h^{*}\ar[r]^{g^{*}\eta(h)}_{\sim}&g^{*}h^{*}f_{Z}\ar[d]_{\wr}\\
          f_{X}(\text{L}(h\cdot g)^{*})\ar[rr]^{\sim}&  &(h\cdot g)^{*}f_{Z}
}$$
commutes where $\theta$ is the canonical isomorphism 
$$\theta:\xymatrix{ \text{L}g^{*}\text{L}h^{*}\ar[r]^{\sim}&\text{L}(h\cdot g)^{*}},$$
iv) On finite complexes of locally free $\mathcal{O}_{X}$-modules
$$f=\det~\text{and}~~~i=\mathrm{i}$$
\end{defi}
\begin{theo}\label{ucd}
 There is one, and, up to canonical isomorphism, only one extended determinant functor $(f, \mathrm{i})$ which we will write $(\det,\mathrm{i})$.
\end{theo}
\begin{proof}
See \cite{Kund}, Theorem 2, Pag. 42.
\end{proof}

The theorem above implies that the functor $(\det, \mathrm{i})$ have same compatibility as ordinary $\det^{*}$. In particular:

a) If each term $\mathcal{F}^{n}$ in the corresponding perfect complex $\mathcal{F}^{\cdot}$ is itself perfect, i.e., has locally a finite free resolution, then $$\det(\mathcal{F}^{\cdot})\cong \otimes_{n}\text{det}^{*}(\mathcal{F}^{n})^{(-1)^{n}}.$$

b) If the cohomology sheaves $H^{n}(\mathcal{F}^{\cdot})$ of the complex are perfect we denote the objects of subcategory by $\text{Parf}^{0}\subset\text{Parf}$,
then $$\det(\mathcal{F}^{\cdot})\cong \otimes_{n}\text{det}^{*}(H^{n}(\mathcal{F}^{\cdot}))^{(-1)^{n}}.$$

  For a vector bundle $E$, if $Rf_{*}E$ is a strictly perfect complex under some suitable morphism where $Rf_{*}$ is viewed as the right derived functor of $f_{*}$, then the properties of extended determinant functor is valid for the strictly perfect complex.

To conclude this section, we put together the determinant functor, the virtual category, and the Picard category  to make the following definition.
\begin{defi}\label{vdk}
For a scheme $X$, we denote by $Vect(X)$ the exact category of vector bundles over $X$ and by $V(X): =V(Vect(X))$ (resp. $V(Y)$) the virtual category of vector bundles on $X$ (resp. $Y$).
Let $f:X\rightarrow Y$ be a smooth and projective morphism and  $Y$ carries an ample invertible sheaf. Then there exists an induced functor from $V(X)$ to the Picard category $\mathscr{P}is_{Y}$ (the definition of $\mathscr{P}is_{Y}$is in the example \ref{lbd}) denoted 
by $\det Rf_{*}$, which is defined as follows:

In the Theorem \ref{ucd}, we have a unique functor $\det: \text{Perf-is}_{Y} \rightarrow \mathscr{P}is_{Y}$ . For any vector bundle $E$ from the exact category $Vect(X)$, it can be viewed a perfect complex $E^{.}$ with a term $E$ at 
degree $0$ and $0$ at other degree.   
For any perfect complex $E^{.}$ of $\mathcal{O}_{X}$-modules and the morphism $f$ in the definition, $Rf_{*}E^{.}$ is still a perfect complex of $\mathcal{O}_{Y}$-modules (see \cite{Berth}, IV, 2.12). Therefore, we have a functor
$\det Rf_{*} :(Vect(X), iso)\rightarrow\text{Perf-is}_{Y}$ where $(Vect(X), iso)$ is the category with the same objects from $Vect(X)$ and morphisms being only isomorphisms.
By the definition of the extended determinant functor $\det$, it can be verified that $\det Rf_{*}$ satisfies the same conditions from a) to d) with $[~]$ in Def. \ref{vircon}.
By the universality of the virtual category $V(X)$, the functor $\det Rf_{*}$ factors uniquely up to unique isomorphism through $(Vect(X), iso)\rightarrow V(X)$. More clearly, we have the following diagram:
$$\xymatrix{(Vect(X), iso)\ar[d]\ar[r]^(.60){Rf_{*}}&\text{Perf-is}_{Y}\ar[d]^{\det}\\
  V(X)\ar@{.>}[r]^{\det Rf_{*}}&\mathscr{P}is_{Y}     
        }$$
Meanwhile, there is a functor
 $V(X)\rightarrow\mathscr{P}is_{Y}$ which is still denoted by $\det Rf_{*}$.
 \end{defi}

\subsection{The Deligne pairing}\label{sdp}
Before stating the functorial Riemann-Roch theorem, it is necessary to introduce the Deligne pairing, which appeared in \cite{Dir} for the first time. Originally, Deligne pairing was defined for a morphism which is proper, flat and purely of relative dimension $1$. The Deligne pairing under the circumstance above is closely related to a norm functor (See \cite{Del} for their relation). In this paper, we need the Deligne pairing for a morphism of relative  dimension $n$ which is defined by S. Zhang. 
\begin{defi}\label{d,p.n}(See \cite{Zhang}, \S 1.1)
Let $f:X\rightarrow S$ be a flat and projective morphism of integral schemes of purely relative dimension $n$.
Let $L_{0},\cdots ,L_{n}$ be line bundles on $X$. Then $\langle L_{0},\cdots ,L_{n}\rangle$ is defined to be the $\mathcal{O}_{S}$-module
which is generated, locally for Zariski topology on $S$, by the symbols $\langle l_{0},\cdots,l_{n}\rangle$ for rational sections $l_{i}$ of $L_{i}$ for 
each $i\in [0, n]$ such that the corresponding divisors are disjoint, i.e., $\bigcap_{i=0}^{n} \text{div}( l_{i}) =\emptyset$ 
 with the following relation: for some $i$ between $0$ and $n$ and rational function $g$ on $X$, if the intersection 
 $\cap_{j\neq i}\text{div}( l_{j})=\sum_{k}n_{k}Y_{k}$ is finite over $S$ and has empty intersection with $\text{div}(f)$, then
 $$\langle l_{0},\cdots,fl_{i},\cdots,l_{n}\rangle=\prod_{k}\text{Norm}_{Y_{k}}(g)^{n_{k}}\langle l_{0},\cdots,l_{i},\cdots,l_{n}\rangle $$
 \end{defi}
By the definition, it is direct that the pairing is symmetric and multi-linear with respect to the group structure defined by the tensor products and dualization. In sequel, we will call the Deligne pairing above extended Deligne pairing.
  
In this subsection, as in section 2, the Picard category of graded line bundles still will be denoted by $\mathscr{P}is_{X}$ and the virtual category of the exact category of vector bundles will be denoted by $V(X)$ for any scheme $X$.
For any vector bundle $E$ from an exact category of vector bundles, which is viewed as a complex, $Rf_{*}E$ is a complex again under some given morphism $f$. The most important property of an extended Deligne pairing is that there exists a canonical isomorphism between the extended Deligne pairing and line bundles under the determinant functor.
\begin{theo}
Let $f:X\rightarrow S$ be projective, flat and purely of relative dimension $n$ and $L_{i}$ be bundles over $X$ with $i \in [0, n]$.
Then there is a canonical isomorphism $\langle L_{0}, \cdots, L_{n}\rangle \cong \det Rf_{*} (\otimes_{i=0}^{n}(L_{i}-\mathcal{O}_{X}))$.
\end{theo}
\begin{proof}
See Theorem 1 in \cite{BSW}.
\end{proof}
In \cite{Del}, Deligne  provided a variant of the Deligne pairing (see Page. 149 , \cite{Del}) which is related to a norm functor. For the extended 
Deligne pairing, we have a property which is completely analogous to the variant of Deligne pairing. The next proposition is essentially a consequence 
of the theorem above. 
\begin{pro}\label{d,p}
Let $f:X\rightarrow S$ be projective, flat and purely of relative dimension $n$.  For each $i\in [0,n]$, let $E_{i}$ and $F_{i}$ be locally free coherent sheaves with same rank everywhere over $X$. Then we have the following canonical isomorphism
$$\langle\det(E_{0}-F_{0}),\cdots, \det(E_{n}-F_{n})\rangle\cong\det Rf_{*}(\otimes_{i=0}^{n}(E_{i}-F_{i})).$$
\end{pro}
\begin{proof}
In the first, observe that  for any line bundles $A$ and $L_{i}$ over $X$, $\langle\det(A-A),L_{1}, \cdots, L_{n}\rangle \cong \langle \mathcal{O}_{X} ,L_{1},
\cdots, L_{n}\rangle \cong \mathcal{O}_{S}$. Therefore,  we have $\langle A^{*},L_{1}, \cdots, L_{n}\rangle\\
\cong \langle A,L_{1}, \cdots, L_{n}\rangle^{*}$,
where the superscript $\ast$ is an operation by taking dualization of a line bundle.  Then for any two line bundles $A, B$ over $X$, according to the Theorem \ref{d,p}, we have
the following isomorphisms:
\begin{align}
\langle\det(A-B),L_{1}, \cdots, L_{n}\rangle &\cong\langle A,L_{1}, \cdots, L_{n}\rangle\otimes \langle B^{*},L_{1}, \cdots, L_{n}\rangle\notag\\
&\cong \langle A,L_{1}, \cdots, L_{n}\rangle\otimes \langle B,L_{1}, \cdots, L_{n}\rangle^{*}\notag\\
&\cong (\det Rf_{*}((A-\mathcal{O}_{X})\otimes_{i=1}^{n}(L_{i}-\mathcal{O}_{X})))\notag\\
 &\otimes(\det Rf_{*}((B-\mathcal{O}_{X})\otimes_{i=1}^{n}(L_{i}-\mathcal{O}_{X})))^{\ast}\notag\\
 &\cong (\det Rf_{*}((A-\mathcal{O}_{X})\otimes_{i=1}^{n}(L_{i}-\mathcal{O}_{X})))\notag\\
 &\otimes\det Rf_{*}((\mathcal{O}_{X}-B)\otimes_{i=1}^{n}(L_{i}-\mathcal{O}_{X}))\notag\\
 &\cong \det Rf_{*}(((A-\mathcal{O}_{X})+(\mathcal{O}_{X}- B))\otimes_{i=1}^{n}(L_{i}-\mathcal{O}_{X}))\notag\\
&\cong\det Rf_{*}((A-B)\otimes_{i=1}^{n}(L_{i}-\mathcal{O}_{X})).\notag
\end{align}
Moreover, by the splitting principle, it is sufficient to prove the proposition by considering $E_{i}-F_{i}=\sum_{k=1}^{i_{s}}(A_{ik}-B_{ik})$
with line bundles $A_{ik}, B_{ik}$ over $X$ for $i\in [0,n]$. Also,  a well-known fact is that the determinant functors in both of sides of the isomorphism in  the proposition turn the additivity into the tensor product. Then one has\begin{align}
\langle\det(E_{0}-F_{0}),L_{1}, \cdots, L_{n}\rangle &\cong \langle\det(\sum_{k=1}^{0_{s}}(A_{0k}-B_{0k})),L_{1}, \cdots, L_{n}\rangle\notag\\ 
&\cong \otimes_{k=1}^{0_{s}}\langle\det((A_{0k}-B_{0k})),L_{1}, \cdots, L_{n}\rangle\notag\\ 
&\cong \otimes_{k=1}^{0_{s}}\det Rf_{*}((A_{0k}-B_{0k})\otimes_{i=1}^{n}(L_{i}-\mathcal{O}_{X}))\notag\\
& \cong \det Rf_{*}(\sum_{k=1}^{0_{s}}(A_{0k}-B_{0k})\otimes_{i=1}^{n}(L_{i}-\mathcal{O}_{X}))\notag\\
& \cong\det Rf_{*}(E_{0}-F_{0}))\otimes_{i=1}^{n}(L_{i}-\mathcal{O}_{X}))\notag
\end{align} 
Finally, our proposition follows from
\begin{align}
\langle\det(E_{0}-F_{0}), &\cdots,\det( E_{n}-F_{n})\rangle\notag\\
&\cong\langle\det(\sum_{k=1}^{0_{s}}(A_{0k}-B_{0k}), \cdots, \det(\sum_{k=1}^{n_{s}}(A_{nk}-B_{nk})\rangle\notag\\
&\cong\otimes_{k_{0},\cdots, k_{n}}^{0_{s},\cdots, n_{s}} \langle\det(A_{0k_{0}}-B_{0k_{0}}), \cdots, \det(A_{nk_{n}}-B_{nk_{n}})\rangle\notag\\
&\cong\otimes_{k_{0},\cdots, k_{n}}^{0_{s},\cdots, n_{s}} \det Rf_{\ast}(\otimes_{i=0}^{n}(A_{ik_{i}}-B_{ik_{i}}))\notag\\
&\cong\det Rf_{\ast}(\otimes_{i=0}^{n}(\sum_{k_{i}=1}^{i_{s}}(A_{ik_{i}}-B_{ik_{i}})))\notag\\
&\cong  \det Rf_{\ast}(\otimes_{i=0}^{n}(E_{i}-F_{i})).\notag
\end{align} 
\end{proof}
\begin{coro}\label{trs} 
In particular, we have a canonical isomorphism $\det Rf_{*}((H_{0}-H_{1})^{\otimes l}\otimes H)\cong\mathcal{O}_{S}$ if $l\geq n+2$ and the ranks
\text{rk}$H_{0} = \text{rk}H_{1}$, for any vector bundles $H_{0}, H_{1}, H$ over $X$ and $f$ as in the theorem, which is stable under base change.
\end{coro}
\begin{proof} 
We will first prove the corollary for $l=n+2$ and then for $l>n+2$. We apply Proposition \ref{d,p} to $E_{0}=H_{0}^{\otimes 2}+H_{1}^{\otimes 2},$
 $F_{0}=2(H_{0}\otimes H_{1})$, $E_{1}=H_{0}\otimes H$, $F_{1}=H_{1}\otimes H$ and $E_{i}=H_{0}$, $F_{i}=H_{1}$ for each $i\in [2, n]$. \\
Then we have the following:
\begin{align}
 &(E_{0}-F_{0})\otimes(E_{1}-F_{1})\otimes _{i=2}^{n}(E_{i}-F_{i})\notag\\
 &\cong (H_{0}-H_{1})^{\otimes 2}\otimes H\otimes(H_{0}-H_{1})\otimes(H_{0}-H_{1})^{\otimes (n-1)}\notag\\
&\cong H \otimes (H_{0}-H_{1})^{\otimes (n+2)}.\notag 
\end{align}
Because of the ranks $\text{rk} (H_{0}) =\text{rk} (H_{1})$, we immediately have equalities of ranks $\text{rk} E_{0}= \text{rk} F_{0}$, $\text{rk}E_{i}=\text{rk} F_{i}$ for each $i\in [1,n]$.
Meanwhile, notice that 
\begin{align}
\det(E_{0}-F_{0})&\cong(\det(H_{0}))^{\otimes 2\text{rk}(H_{0})}\otimes(\det(H_{1}))^{\otimes 2\text{rk}(H_{1})}\notag\\
&\otimes((\det(H_{0})^{\otimes -\text{rk}(H_{0})}\otimes(\det(H_{1}))^{\otimes -\text{rk}(H_{1})})^{2}\notag\\
&\cong\mathcal{O}_{X}.\notag
\end{align}

According to the multi-linear of the Deligne pairing ( see statements after Def. \ref{d,p.n}), by a trivial computation for line bundles $L_{i}$ over $X$ with $i\in [1,n]$:
$\langle\mathcal{O}_{X},L_{1},\cdots ,L_{n}\rangle\cong\langle\mathcal{O}_{X}\otimes\mathcal{O}_{X},L_{1},\cdots ,L_{n}\rangle\cong\langle\mathcal{O}_{X},L_{1},\cdots ,L_{n}\rangle\otimes\langle\mathcal{O}_{X},L_{1},\cdots ,L_{n}\rangle$, the obvious consequence is \\
$\langle\mathcal{O}_{X},L_{1},\cdots ,L_{n}\rangle \cong \mathcal{O}_{S}$.
Now, we can obtain the corollary by 
\begin{align}
\det Rf_{*}((H_{0}-H_{1})^{\otimes (n+2)}\otimes H)&\cong\det Rf_{*}((E_{0}-F_{0})\otimes_{i=1}^{n} (E_{i}-F_{i}))\notag\\
&\cong\langle(\det(E_{0}-F_{0}),\det(E_{1}-F_{1}),\cdots,\det(E_{n}-F_{n})\rangle\notag\\
&\cong\langle\mathcal{O}_{X},\det(E_{1}-F_{1}),\cdots, \det(E_{n}-F_{n})\rangle\notag\\
&\cong \mathcal{O}_{S}.\notag
\end{align}
This proves the corollary for the case $l=n+2$.
In general, for any integer $l>0$, one has the virtual bundle $(E_{1}-F_{1})\otimes (H_{0}-H_{1})^{\otimes l}$ of rank $0$. By the same reason as above,
these are the following isomorphisms:
\begin{align}
\det &Rf_{*}((H_{0}-H_{1})^{\otimes l}\otimes (H_{0}-H_{1})^{\otimes (n+2)}\otimes H)\notag\\
&\cong\det Rf_{*}((E_{0}-F_{0})\otimes (E_{1}-F_{1})\otimes (H_{0}-H_{1})^{\otimes l}\otimes_{i=2}^{n}(E_{i}-F_{i}))\notag\\
&\cong\langle(\det(E_{0}-F_{0}),\det((E_{1}-F_{1})\otimes (H_{0}-H_{1})^{\otimes l}),\det (E_{2}-F_{2}),\cdots, \det (E_{n}-F_{n})\rangle\notag\\
&\cong\langle\mathcal{O}_{X},\det((E_{1}-F_{1})\otimes (H_{0}-H_{1})^{\otimes l}),\det (E_{2}-F_{2}),\cdots, \det (E_{n}-F_{n})\rangle\notag\\
&\cong \mathcal{O}_{S}.\notag
\end{align}
Therefore, one has $\det Rf_{*}((H_{0}-H_{1})^{\otimes l}\otimes H)\cong \mathcal{O}_{S} $ for any $l\geqslant n+2$.

For any morphism $g :S^{'}\rightarrow S$, we have the fiber product under base change:
$$\xymatrix{X^{'}\ar[r]^{g^{'}}\ar[d]_{f^{'}}&X\ar[d]^{f}\\
S^{'}\ar[r]^{g}&S}$$

Furthermore, the projective morphism and the flat morphism are stable under base change (see \cite{Liu}, Chapt. 3), i.e., the morphism $f^{'}$ is flat of relative dimension $n$ and projective. 
For any vector bundle $F$ over $X$, then we have the isomorphism: $$g^{*}(\text{det}_{S}(Rf_{*}F))\cong{det}_{S^{'}}(\text{L}g^{*}(Rf_{*}F))\cong\text{det}_{S^{'}}(Rf^{'}_{*}(g^{'*}F)$$ 
which follows from the iii) of the Def. 2.9 of the extended determinant functor and the formation of cohomology commuting with base change (Again, we will  prove the two isomorphisms in (10) and (11) of of Theorem \ref{functor} (II)).

Let $F$ be $(H_{0}-H_{1})^{\otimes l}\otimes H$ with $l\geqslant n+2$, and the isomorphism above becomes $$ \text{det}_{S^{'}}(Rf^{'}_{*}(g^{'*}F))\cong g^{*}(\text{det}_{S}(Rf_{*}F))\cong g^{*}(\mathcal{O}_{S})\cong \mathcal{O}_{S^{'}}.$$ Therefore, we finish the proof.

\end{proof}
\section{The Adams-Riemann-Roch Theorem}
\subsection{The Bott element and the Adams operation}
In this subsection, we will provide the statement of Adams-Riemann-Roch theorem after recalling the ingredients in this theorem, including
the Bott element and the Adams operation.
\begin{defi}\label{ad,2}
For any integer $k\geq 1$, the symbol $\theta^{k}$ refers to an operation, which associates an element of $K_{0}(X)$ to any locally free coherent sheaf on a quasi-compact scheme $X$. It satisfies the following three properties: 

(1) For any invertible sheaf $L$ over $X$, we have

$$\theta^{k}(L)=1+L+\cdots +L^{\otimes k-1};$$

(2) For any short exact sequence $0\longrightarrow E^{'}\longrightarrow E\longrightarrow E^{''}\longrightarrow 0$ of locally free coherent
sheaves on $X$ we have

         $$\theta^{k}(E)=\theta^{k}(E^{'})\otimes\theta^{k} (E^{''});$$

(3) For any morphism of schemes $g : X^{'}\longrightarrow X$ and any locally free coherent sheaf $E$ over $X$ we have

                  $$g^{*}(\theta^{k}(E))=\theta^{k}(g^{*}(E)).$$

If $E$ is a locally free coherent sheaf on a quasi-compact scheme $X$, then the element $\theta^{k}(E)$ is often called the $k$-th Bott element.
\end{defi}

\begin{pro}
There exists a unique operation $\theta^{k}$ satisfying properties (1)-(3) above.
\end{pro}
\begin{proof}
See \cite{Man}. Lemma 16.2. Subsection 16, or SGA, VII.
\end{proof}
\begin{pro}\label{ad,1}
For any scheme $X$ and a positive integer $k\geq 1$, the $k$-th Adams operation is the functorial endomorphism $\psi^{k}_{X}$  (Usually, $\psi^{k}_{X}$ will be denoted by $\psi^{k}$ if there is no confuse) of unitary ring $K_{0}(X)$ which is uniquely determined by the following two conditions,

(1) $\psi^{k}_{X}f^{*}=f^{*}\psi^{k}_{Y}$ for any morphism of Noetherian schemes $f:X\longrightarrow Y$. 

(2) For any invertible sheaf $L$ over $X$, $\psi^{k}(L)=L^{\otimes k}$.
\end{pro}
\begin{proof}
See \cite{Man}, Subsection 11.
\end{proof}
\begin{theo}(Adams-Riemann Roch theorem.)\label{Arr}
We assume that $f:X\rightarrow Y$ is a projective local complete intersection morphism and $Y$ is a quasi-projective scheme 
 over an affine scheme.  Let $\Omega_{f}$ be the relative differentials of $f$.
Then the equality $$\psi^{k}(R^{\bullet}f_{*}(E))=R^{\bullet}f_{*}(\theta ^{k}(\Omega_{f})^{-1}\otimes\psi^{k}(E))$$ for any integer $k\geq 1$ holds
in $K_{0}(Y)\otimes \mathbb{Z}[\frac{1}{k}]$.
\end{theo}
\begin{proof}   See \cite{FL}, V. Th. 7.6.
\end{proof}

\subsection{The Adams Riemann Roch Theorem in positive characteristic}\label{section ARR char p}
In this part, we will provide all the necessary knowledge for the Adams-Riemann-Roch theorem in positive characteristic. But we will not repeat the proof 
which is known. The all ingredient will be used in Section 4.

On a quasi-compact scheme of characteristic $p> 0$, Pink and R\"{o}ssler constructed an explicit representative of the $p$-th Bott element (see \cite{Pi}, Sect. 2). 

We recall the construction:

Let $p$ be a prime number and $Z$ a scheme of characteristic $p$. Let $E$ be a locally free coherent sheaf $Z$. For any integer $k\geq 0$ let
$\text{Sym}^{k}(E)$ denote the $k$-th symmetric power of $E$. Then 
$$\text{Sym}(E):=\bigoplus _{k\geq 0}\text{Sym}^{k}(E)$$ is quasi-coherent graded $\mathcal{O}_{Z}$-algebra, called the symmetric algebra of $E$. Let
$\mathcal{J}_{E}$ denote the graded sheaf of ideals of $\text{Sym}(E)$ that is locally generated by the sections $e^{p}$ of $\text{Sym}^{p}(E)$ for
all sections $e$ of $E$, and set $$\tau(E):=\text{Sym}(E)/\mathcal{J}_{E}.$$
Locally this construction means the following. Consider an open subset $U\subset Z$  such that $E|_{U}$ is free, and
choose a basis $e_{1},\ldots, e_{r}$. Then $\text{Sym}(E)|_{U}$ is the polynomial algebra over $\mathcal{O}_{Z}$ in the variables $e_{1},\ldots,e_{r}$.
Since $Z$ has characteristic $p$, for any open subset $V\subset U$ and any sections $a_{1},\ldots,a_{r} \in \mathcal{O}_{Z}(V)$ we have 
$$(a_{1}e_{1}+\ldots +a_{r}e_{r})^{p}=a_{1}^{p}e_{1}^{p}+\ldots +a_{r}^{p}e_{r}^{p}.$$
It follows that $\mathcal{J}_{E}|_{U}$ is the sheaf of ideals of $\text{Sym}(E)|_{U}$ that  is generated by $e_{1}^{p}\ldots,e_{r}^{p}$. Clearly that description is independent of
the choice of basis and compatible with localization; hence it can be used as an equivalent definition of $\mathcal{J}_{E}$ and $\tau(E)$.
The local description also implies that $\tau(E)|_{U}$ is free over $\mathcal{Z}|_{U}$ with the basis the images of the monomials $e_{1}^{i_{1}}\cdots e_{r}^{i_{r}}$ for all choices of exponents $0\leqq i_{j}< p$. 

It can be showed that $\tau(E)$ satisfies the defining properties of the $p$-th Bott element. In other words, we have the following proposition (see \cite{Pi}, Prop. 2.6).

\begin{pro}\label{btp}
For any locally free coherent sheaf $E$ of finite rank on a scheme $Z$ of characteristic $p>0$, we have $\tau(E)=\theta ^{p}(E)$ in $K_{0}(Z)$.
\end{pro}

In the case of positive characteristic, we introduce the absolute Frobenius morphism and the relative Frobenius morphism which play a vital role in the proof of the Adams-Riemman-Roch theorem. 
The absolute Frobenius morphism is a morphism of schemes of positive characteristic from a scheme $X$ to itself, usually denoted by 
 $F_{X}$. On topological space, it is identity map and it is locally given $x\rightarrow x^{p}$ on ringed sheaves (See \cite{Liu}). Moreover, 
given a morphism of schemes of positive characteristic, say $f: X\rightarrow Y$, then there is  a commutative diagram from the morphism 
$f$
$$\xymatrix{X\ar@/_/[ddr]_{f}\ar@{.>}[dr]|{F_{X/Y}}\ar@/^/[rrd]^{F_{X}}&  & \\
                            & X^{'}\ar[r]_{J}\ar[d]^{f^{'}}&X\ar[d]^{f}    \\
                             &Y\ar[r]^{F_{Y}}&Y
                                 }~~~~~(4)$$
 where $ F_{X}$ and $F_{Y}$ are obvious absolute Frobenius morphisms respectively and $X^{'}$ is a the fiber product of the structure morphism by base change $F_{Y}$. Furthermore,  in the diagram above, $F_{X/Y}$ is called the relative Frobenius morphism of $f$, 
 which comes from the universality of the fiber product $X^{'}$ since the external diagram of the diagram (4) is obviously commutative.   We also denote the relative Frobenius morphism of the morphism $f$ by $F:=F_{X/Y}$ for
simplicity.  We will use these notations in the following propositions and proofs until the end of the paper,.

The following lemma basically due to Kunz describes the structure morphism of schemes how to determine the relative Frobenius morphism:
\begin{lem}\label{l,f} Let $S$ be a scheme of positive characteristic (say $p$) and $f:X\rightarrow S$ a smooth morphism of pure relative dimension $n$. Then the relative Frobenius $F_{X/S}$ is finite and flat, and the $\mathcal{O}_{X^{'}}-$algebra $(F_{X/S})_{*}\mathcal{O}_{X}$ is locally free of rank $p^{n}$. In particular, if $f$ is an \'{e}tale
morphism, then $F_{X/S}$ is an isomorphism.
\end{lem}

The lemma is well-known and featured extensively in the literature. However, its proof is usually not provided. I provide a complete proof by collecting all related known facts in the preprint (See \cite{Qxu}) (or see Thm15.7, \cite{Kunz}).

Since the pull-back $F^{*}$ is adjoint to $F_{*}$ (see \cite{Hart}, Page 110), there is a natural morphism of $\mathcal{O}_{X}$-algebras $F^{*}F_{*}\mathcal{O}_{X}\rightarrow\mathcal{O}_{X}$. Let $I$
be the kernel of the natural morphism. In \cite{Pi}, the following definition is made
$$Gr(F^{*}F_{*}\mathcal{O}_{X}):=\bigoplus_ {k\geq 0}I^{k}/ I^{k+1}$$
which is the associated graded sheaf of $\mathcal{O}_{X}$-algebras. Let $\Omega_{f}$ be the sheaf of relative differentials of $f$.
Also, they proved the following key proposition which can be used to prove the $p$-th Adams-Riemann-Roch theorem in positive characteristic (see \cite{Pi}, Prop. 3.2).
\begin{pro}\label{gci}
There is a natural isomorphism of $\mathcal{O}_{X}$-modules
$$ I/ I^{2}\cong \Omega_{f}$$
and a natural isomorphism of graded  $\mathcal{O}_{X}$-algebras
$$\tau(I/ I^{2})\cong Gr(F^{*}F_{*}\mathcal{O}_{X})$$
 \end{pro}

According to the proposition above, directly there are isomorphisms $$Gr(F^{*}F_{*}\mathcal{O}_{X})\cong\tau(I/ I^{2})\cong \tau(\Omega_{f}).$$ 
Moreover, Proposition \ref{btp} also
implies $\tau(\Omega_{f})\cong \theta ^{p}(\Omega_{f})$. In the Grothendieck groups, we have $Gr(F^{*}F_{*}\mathcal{O}_{X})\cong F^{*}F_{*}\mathcal{O}_{X}$, then the equality $F^{*}F_{*}\mathcal{O}_{X}= \theta ^{p}(\Omega_{f})$ holds by viewing them as elements of the Grothendieck groups. More important,  it also means that $F^{*}F_{*}\mathcal{O}_{X}\cong \theta ^{p}(\Omega_{f})$ is true in the virtual category $V(X)$, which is really needed in our main result-Thm. \ref{mrs}.

In the case of characteristic $p>0$, the Adams operation $\psi^{p}$ can be explicitly replaced  by the pull-back of the absolute Frobenius morphism. This is the following proposition:
\begin{pro}\label{frebad}
  For a scheme $Z$ of characteristic $p>0$ and its absolute Frobenius morphism $F_{Z}: Z\rightarrow Z$,  then the pullback $F_{Z}^{*}: K_{0}(Z)\rightarrow K_{0}(Z)$ is  just the $p$-th Adams operation $\psi^{p}$. 
\end{pro}
\begin{proof}
 This is a well-known fact (see \cite{Bern 1}, Pag. 64, Proposition 2.15), which is also a consequence of the splitting principle (see \cite{Man}, Par. 5).
\end{proof}
When $Y$ is in addition a scheme of characteristic $p>0$, Pink and R\"ossler have given a short proof of Theorem \ref{Arr} using the tools introduced in this 
section, which is the following:
\begin{theo}
The Adams-Riemann-Roch theorem is true under the assumption given in the beginning of this subsection i.e., the following equality 
$$\psi^{p}(R^{\bullet}f_{*}(E))=R^{\bullet}f_{*}(\theta ^{p}(\Omega_{f})^{-1}\otimes\psi^{p}(E)) $$
holds in $K_{0}(Y)[\frac{1}{p}]:=K_{0}(Y)\otimes_{\mathbb{Z}}\mathbb{Z}[\frac{1}{p}]$.
\end{theo}
\begin{proof} See \cite{Pi}, page 1074. 
\end{proof}
\begin{rem} In\cite{Pi}, there is an inverse for a vector bundle $E$ in $K(X)_{\mathbb{Q}}$ of a regular scheme $X$ because the Grothendieck-$\gamma$ filtration is finer than the topological filtration (see Chap.III \cite{FL}) so that for any vector bundle $E$ of rank $r$, $E-r$ is nilpotent or $E$ is invertible in $K(X)_{\mathbb{Q}}$. Formally, if we denote $E$ by $x$ in $K(X)_{\mathbb{Q}}$, then we have 
$\frac{1}{x}=\frac{1}{r-(r-x)}=\frac{1}{r(1-\frac{r-x}{r})}=\frac{1}{r}+\frac{r-x}{r^{2}}+\frac{(r-x)^{2}}{r^{3}}+\cdots \frac{(r-x)^{n}}{r^{n+1}}+\cdots $. But for some $n$, $(r-x)^{n}$ vanishes, so the $\frac{1}{x}$ is equal to a sum of finite terms $(r-x)^{i}$ with $i<n$ in $K(X)_{\mathbb{Q}}$ (Also, see \cite{Berth} VII).
\end{rem}
\section{Main Results}
We have already introduced all the necessary ingredients for our theorem.
\begin{theo} \label{mrs} $f:X\rightarrow S$ be projective and smooth of relative dimension $n$, where $S$ is a quasi-compact scheme of characteristic $p>0$ and carries an ample invertible sheaf. Let $E$ be a vector bundle over $X$ and $\Omega_{f}$ be the sheaf of relative 
differentials of the morphism $f$.  

 (I) Then we have a determinant version of Adams-Riemann-Roch theorem
\begin{align}
 \psi^{p}((\det Rf_{*}E)^{\otimes p^{n(n+2)}})\cong&\det Rf_{*}(\widetilde{\theta^{p}(\Omega_{f})^{-1}}\otimes \psi^{p}(E))\notag
\end{align}
where we denote $\widetilde{\theta^{p}(\Omega_{f})^{-1}}$ by $(-1)^{(n+1)}(\sum_{i=0}^{n+1}\binom{n+2}{i}(\theta^{p}(\Omega_{f}))^{\otimes (n+1-i)}(-p^{n})^{i})$.

(II) The isomorphism in (I) is functorial, i.e. , if $g :S^{'}\rightarrow S$ is a base extension and the corresponding fiber product is $X^{'}$:
$$\xymatrix{X^{'}\ar[r]^{g^{'}}\ar[d]_{f^{'}}&X\ar[d]^{f}\\
S^{'}\ar[r]^{g}&S },$$
then there are canonical isomorphisms over $S^{'}$:
$$\xymatrix{g^{*}(\psi^{p}(\text{det}_{S}(Rf_{*}(E))^{\otimes p^{3}})\ar[rr]^{\cong}\ar[d]_{\cong}& &\psi^{p}(\det_{S^{'}} Rf^{'}_{*}g^{'*}E)^{\otimes p^{3}}\ar@{.>}[d]^{\cong}\\
  A\ar[rr]^{\cong}& &B}$$
where \begin{align} 
A=: &g^{*}(\det Rf_{*}(\widetilde{\theta^{p}(\Omega_{f})^{-1}}\otimes \psi^{p}(E)))\notag
\end{align}
\begin{align}
 B=: &\det Rf^{'}_{*}(\widetilde{\theta^{p}(\Omega_{f^{'}})^{-1}}\otimes \psi^{p}(g^{'*}E))\notag
\end{align}
and $\Omega_{f^{'}}$ is the relative differentials of the morphism $f^{'}$.
\end{theo}
According to Definition \ref{vdk}, there is an induced functor from the virtual category $V(X)$ to the Picard category $\mathscr{P}is_{S}$. The isomorphisms in (I) and (II) can be viewed as 
isomorphisms of line bundles even though we did not write out the degree of graded line bundles, but that is from the isomorphism of graded line bundles in the category $\mathscr{P}is_{S}$. 
Because for any two objects $(L,l)$ and $(M,m)$ in the category $\mathscr{P}is_{S}$, they are isomorphic if and only if $L\cong M$ and $l=m$. We will apply ideas appearing in  the proof of the $p$-th Adams-Riemann-Roch Theorem in the case of characteristic $p>0$ to our proof. 
\begin{proof}
 In (I), for any prime number $p$, we have the following isomorphisms.
 Here, we explain them one by one. We will use the following diagram and some notations again:
$$\xymatrix{X\ar@/_/[ddr]_{f}\ar@{.>}[dr]|{F_{X/S}}\ar @/^/[rrd]^{F_{X}}&  & \\
                            & X^{'}\ar[r]_{J}\ar[d]^{f^{'}}&X\ar[d]^{f}\\
                             &S\ar[r]^{F_{S}}&S
                                 }$$
                                                                
We will continue to use $F$ to denote the relative Frobenius instead of $F_{X/S}$ for simplicity.

\begin{align}
&\psi^{p}(\det Rf_{*}E)^{\otimes p^{n(n+2)}}\notag\\
&\cong F^{*}_{S}(\det (Rf_{*}E))^{\otimes p^{n(n+2)}}\notag\\
&\cong(\det \text{L}F^{*}_{S}(Rf_{*}E))^{\otimes p^{n(n+2)}}\\
&\cong(\det Rf^{'}_{*}(J^{*}E))^{\otimes p^{n(n+2)}}\\
&\cong\det Rf^{'}_{*}(p^{n(n+2)}J^{*}E+(-1)^{(n+1)}(F_{*}\mathcal{O}_{X}-p^{n})^{\otimes (n+2)}\otimes J^{*}E)\\
&\cong\det Rf^{'}_{*}(p^{n(n+2)}J^{*}E+(-1)^{(n+1)}(\sum_{i=0}^{n+1}\binom{n+2}{i}(F_{*}\mathcal{O}_{X})^{\otimes (n+2-i)}(-p^{n})^{i})\otimes J^{*}E\notag\\
&- (p^{n})^{(n+2)}J^{*}E)\\
&\cong\det Rf^{'}_{*}((-1)^{(n+1)}(F_{*}\mathcal{O}_{X})\otimes(\sum_{i=0}^{n+1}\binom{n+2}{i}(F_{*}\mathcal{O}_{X})^{\otimes (n+1-i)}(-p^{n})^{i})\otimes J^{*}E)\\
&\cong\det Rf_{*}(F^{*}((-1)^{(n+1)}(\sum_{i=0}^{n+1}\binom{n+2}{i}(F_{*}\mathcal{O}_{X})^{\otimes (n+1-i)}(-p^{n})^{i})\otimes J^{*}E))\\
&\cong\det Rf_{*}((-1)^{(n+1)}(\sum_{i=0}^{n+1}\binom{n+2}{i}(F^{*}F_{*}\mathcal{O}_{X})^{\otimes (n+1-i)}(-p^{n})^{i})\otimes F^{*}J^{*}E)\\
&\cong\det Rf_{*}((-1)^{(n+1)}(\sum_{i=0}^{n+1}\binom{n+2}{i}(\theta^{p}(\Omega_{f}))^{\otimes (n+1-i)}(-p^{n})^{i})\otimes \psi^{p}(E))\\
&\cong\det Rf_{*}(\widetilde{\theta^{p}(\Omega_{f})^{-1}}\otimes \psi^{p}(E))
\end{align}

Now, we can achieve our result by illustrating these isomorphisms above one by one. Firstly, $p$-th Adams operation $\psi^{p}$ can be replaced by $F^{*}_{S}$, which is from Prop. \ref{frebad}.

By the definition of the extended determinant functor, the extended determinant functor commutes with the pull-back (see Definition \ref{exd}, item iii)), which is the
isomorphism (1).

For (2), we know that the formation of cohomology commutes with base change, i.e., $\text{L}F^{*}_{S}\cdot Rf_{*}\cong Rf^{'}_{*}\cdot\text{L}J^{*}$ (see \cite{Berth}, IV, Prop. 3.1.1), where $\text{L}F^{*}_{S}$ is the left derived functor of the functor $F^{*}_{S}$. Because $E$ is a vector bundle, $\text{L}F^{*}_{S}$ is the same with $F^{*}_{S}$

In (3), we introduce a new term $(-1)^{n+1}(F_{*}\mathcal{O}_{X}-p^{n})^{\otimes (n+2)}\otimes J^{*}E$. In Lemma \ref{l,f}, we know that $F_{*}\mathcal{O}_{X}$ is locally free of rank $p^{r}$ where $r$ is the relative dimension 
of the morphism $f$. Because $f$ is relatively smooth of dimension $n$ in our condition, $F_{*}\mathcal{O}_{X} -p^{n}$ is a virtual vector bundle of rank $0$.
According to Corollary \ref{trs}, $\det Rf^{'}_{*}((-1)^{n+1}(F_{*}\mathcal{O}_{X}-p^{n})^{\otimes (n+2)}\otimes J^{*}E)$ is trivial.

Following (3), (4) is a expansion of (3) and (5) is a recombination of (4), by making the term  $(p^{n})^{(n+2)}J^{*}E$ vanish and taking the term $F_{*}\mathcal{O}_{X}$ out from the corresponding binomial sum. 

(6) is direct from the projection formula (see \cite{Berth}, III, Pro. 3.7) and the fact $F^{*}_{X}=F^{*}J^{*}$. We know that $F_{X}^{*}$ has the same property with the $\psi^{p}$ in the case of characteristic $p>0$ by Prop. \ref{frebad}, i.e., $F_{X}^{*}(L)=L^{\otimes p}$.

When we use the property that $F^{*}_{S}$ is the same with $p$-th Adams operation $\psi^{p}$ again,  (7) is only a recombination of (6). In the Adams-Riemann-Roch theorem in characteristic $p>0$,  we have the isomorphism $F^{*}F_{*}\mathcal{O}_{X}\cong \theta^{p}(\Omega_{f})$ (see Prop. \ref{gci} and some statements before Prop . \ref{ad,1}).  Therefore, we have the isomorphism (8) by replacing $F^{*}F_{*}\mathcal{O}_{X}$ by $ \theta^{p}(\Omega_{f})$. By our notation $\widetilde{\theta^{p}(\Omega_{f})^{-1}}$,  this is just the isomorphism  (9), which finishes the proof of (I).

For (II), these isomorphisms follow almost from the definition of the extended determinant functor and some well-known facts about base-change.

The left vertical isomorphism in the diagram of (II) is direct by the pull-back of the isomorphism (I), which means that one make the pull-back for two sides of 
the isomorphism (I), i.e., 
$$g^{*}( \psi^{p}((\det Rf_{*}E)^{\otimes p^{n(n+2)}})\cong g^{*}(\det Rf_{*}(\widetilde{\theta^{p}(\Omega_{f})^{-1}}\otimes \psi^{p}(E)))=A.$$
For the upper horizontal isomorphism, this is almost to repeat proof of (I).
\begin{align}&g^{*}(\psi^{p}(\text{det}_{S}(Rf_{*}(E))^{\otimes p^{n(n+2)}})\notag\\
                    &\cong g^{*}((\text{det}_{S}(Rf_{*}(E)))^{\otimes p^{n(n+2)+1}})\\
                    &\cong(\det \text{L}g^{'*}_{S}(Rf_{*}E))^{\otimes p^{n(n+2)+1}}\\
                    &\cong(\det Rf^{'}_{*}(g^{'*}E))^{\otimes p^{n(n+2)+1}} \\
                    &\cong\psi^{p}(\det Rf^{'}_{*}(g^{'*}E))^{\otimes p^{n(n+2)}}       
                     \end{align}
Because under the extended determinant functor, one alway has line bundles $\text{det}_{S}(Rf_{*}(E)$ and $(\det Rf^{'}_{*}(g^{'*}E)$.
According to the definition of $\psi^{p}$, $\psi^{p}(L)=L^{\otimes p}$ for a line bundle $L$, this is (10) and (13).
By the Definition  \ref{exd}, item iii) of extended determinant functor, extended determinant functor commutes with the pull-ball of any morphism, which is the isomorphism (11).  As in (2), the formation of cohomology commutes with base change, $g^{*}\text{det}_{S}\cong \det \text{L}g^{'*}Rf_{*}$. Hence, this is (12). So one has upper horizontal isomorphism.
Moreover, 

\begin{align} 
A&= g^{*}(\det Rf_{*}(\widetilde{\theta^{p}(\Omega_{f})^{-1}}\otimes \psi^{p}(E)))\notag\\
&\cong \det \text{L}g^{'*}Rf_{*}(\widetilde{\theta^{p}(\Omega_{f})^{-1}}\otimes \psi^{p}(E))\notag\\
&\cong\det \text{L}g^{'*}Rf_{*}(\widetilde{\theta^{p}(\Omega_{f})^{-1}}\otimes \psi^{p}(E))\notag\\
&\cong\det Rf^{'}_{*}(g^{'*}(\widetilde{\theta^{p}(\Omega_{f})^{-1}}\otimes \psi^{p}(E)))\notag\\
&\cong\det Rf^{'}_{*}(g^{'*}((-1)^{(n+1)}(\sum_{i=0}^{n+1}\binom{n+2}{i}\theta^{p}(\Omega_{f})^{\otimes (n+1-i)}(-p^{n})^{i}))\otimes \psi^{p}(g^{'*}E))\notag\\
&\cong\det Rf^{'}_{*}(((-1)^{(n+1)}(\sum_{i=0}^{n+1}\binom{n+2}{i}\theta^{p}(\Omega_{f^{'}})^{\otimes (n+1-i)}(-p^{n})^{i})^{p})\otimes \psi^{p}(g^{'*}E))\notag\\
&\cong\det Rf^{'}_{*}(\widetilde{\theta^{p}(\Omega_{f^{'}})^{-1}}\otimes\psi^{p}(g^{'*}E))\notag\\
&=B\notag
\end{align}
By the composition of the upper horizontal isomorphism, the left vertical isomorphism and the lower horizontal isomorphism,  
one has the right vertical isomorphism, i.e., $\psi^{p}(\det Rf^{'}_{*}(g^{'*}E))^{\otimes p^{n(n+2)}}\cong B$. One observes that 
\\$\psi^{p}(\det Rf^{'}_{*}(g^{'*}E))^{\otimes p^{n(n+2)}}\cong B$ is just from the morphism $f^{'}$, which has a same type with the isomorphism in (I) from the morphism $f$. That is just the functoriality of the isomorphism (I).
\end{proof}
\begin{rem}
In our theorem, there is no higher term of power larger than $n+1$ for  the term $\theta^{p}(\Omega_{f})$ because we have Cor.\ref{trs}, i.e., $\det Rf^{'}_{*}((F_{*}\mathcal{O}_{X}-p^{n})^{\otimes (n+2)}\otimes J^{*}E)$ is trivial. By some combinatorial technique,  $p^{n(n+2)}J^{*}E$ is killed in $(F_{*}\mathcal{O}_{X}-p^{n})^{\otimes n+2}\otimes J^{*}E$. In the remaining term $$(-1)^{(n+1)}(\sum_{i=0}^{n+1}\binom{n+2}{i}(F_{*}\mathcal{O}_{X})^{\otimes (n+2-i)}(-p^{n})^{i})\otimes J^{*}E$$, one can extract a term $F_{*}\mathcal{O}_{X}$ such 
that $(-1)^{(n+1)}(\sum_{i=0}^{n+1}\binom{n+2}{i}(F_{*}\mathcal{O}_{X})^{\otimes (n+2-i)}(-p^{n})^{i})\otimes J^{*}E\cong F_{*}\mathcal{O}_{X}
\otimes (-1)^{(n+1)}(\sum_{i=0}^{n+1}\binom{n+2}{i}(F_{*}\mathcal{O}_{X})^{\otimes (n+1-i)}(-p^{n})^{i})\otimes J^{*}E$. By projection formula under the determinant functor,  only remaining term is \\
$(-1)^{(n+1)}(\sum_{i=0}^{n+1}\binom{n+2}{i}(F^{*}F_{*}\mathcal{O}_{X})^{\otimes (n+1-i)}(-p^{n})^{i})\otimes \psi^{p}(E)$. Therefore,  the term $(-1)^{(n+1)}(\sum_{i=0}^{n+1}\binom{n+2}{i}(F^{*}F_{*}\mathcal{O}_{X})^{\otimes (n+1-i)}(-p^{n})^{i})$ can be thought as  the inverse of $\theta^{p}(\Omega_{f})$, denoted by $\widetilde{\theta^{p}(\Omega_{f})^{-1}}$ rather than its real inverse.
\end{rem}

\begin{coro}\label{functor}
In particular, under the conditions in the theorem above with the relative dimension $n=1$, if one takes the the vector bundle $E$ to be a line bundle $L$ and $p=2$, then one has 
 the functorial Riemann Roch theorem 
 \begin{align}
(\det &Rf_{*}L)^{\otimes 18 }\notag\\
&\cong (\det  Rf_{*}\mathcal{O})^{\otimes 18 }\otimes (\det  Rf_{*}(L^{\otimes 2}\otimes 
\omega^{-1}))^{\otimes 6}\otimes  (\det  Rf_{*}(L\otimes \omega^{-1}))^{\otimes (-6)}\notag
\end{align} 
which coincides with the statement of Deligne's functorial Riemann Roch theorem and the Mumford isomorphism: 
 $\det Rf_{*}(\omega^{n})\cong(\det Rf_{*}
(\omega))^{6n^{2}-6n+1}$, i.e., $\lambda_{n}=\lambda_{1}^{6n^{2}-6n+1}$ if we denote $\det Rf_{*}\mathcal{\omega}^{\otimes n}$ by $\lambda_{n}$ for $n\geqslant 0$,where $\omega$ is the canonical sheaf of the morphism $f$.
\end{coro}
\begin{proof}
This corollary is one of the main results of \cite{Qxu}. See Cor. 4.10, in \cite{Qxu}.
\end{proof}
\begin{rem}
Our theorem is not a consequence of the Adams-Riemann-Roch theorem in $K$-theory. It results from the virtual category and the Picard category, which allows that our theorem is functorial.  In \cite{Den}, D. Eriksson defined the Adams operation and the Bott class on the virtual category and proved that the Adams Riemann Roch theorem was true in the localized Picard category. Before formulating his result, he needs to define what the localized virtual category and the localized Picard category are. These definitions and proofs are impossible to state clearly in several pages. In our theorem, the Adams operation and the Bott class defined on the virtual category are not necessary. We emphasize more about the merits of the positive characteristic, which is one of our motivations.
\end{rem}
Another application of our theorem is the Knudsen-Mumford extension, which is given in \cite{Kund}. We make an introduction to the isomorphism as follows:\\
\textbf{Knudsen-Mumford isomorphism}: let $\pi: X\rightarrow B$ be a flat proper morphism of integral schemes with constant relative dimension $n$,
and let $L$ be a relatively ample line bundle over $X$. The Knudsen-Mumford extension says that there exist functorially defined line bundles 
$\lambda _{i}=\lambda_{i}(X, L, B)$ over $B$ with property:
$$\det \pi_{\ast}(L^{k})\cong \lambda_{n+1}^{\binom{k}{n+1}}\otimes \lambda_{n}^{\binom{k}{n}}\otimes\cdots \otimes \lambda_{0} \text{ for}~k>> 0.$$
For the relative dimension $n=1$, Deligne showed that $ \lambda_{2}=\langle L, L \rangle$, the Delige pairing of $L$ with itself. If in addition the varieties
$X$ and $B$ and are smooth, Deligne proved that $\lambda_{1}(L, X,B)^{2}=\langle LK^{-1}, L\rangle$, where $K=K_{X/B}=K_{X}\otimes K_{B}^{-1}$ is the relative canonical line bundle. For any relative dimension $n\geqslant 1$, we give a somehow different analogue in positive characteristic as our corollary.
\begin{coro}\label{kme}
In particular, when the vector bundle $E$ is a line bundle $L^{\otimes p}$  in Theorem \ref{mrs}, we have an analogue with the Knudsen-Mumford extension in positive characteristic, i.e.,\\
$(\det Rf_{*}L^{\otimes p})^{\otimes p^{n(n+2)+1}}\cong \lambda_{n+1}^{\binom{n+2}{n+1}(p^{n})^{n+1}}\otimes\lambda_{n}^{\binom{n+2}{n}(p^{n})^{n}} \otimes\cdots\otimes\lambda_{0},$ 
where $\lambda_{i}=\det Rf_{*}((-1)^{(n+1+i)}\theta^{p}(\Omega_{f})^{\otimes (n+1-i)}\otimes L^{\otimes p^{2}})$.
\end{coro}
\begin{proof}
In our theorem,  by replacing $E$ by $L^{\otimes p}$, we have :
\begin{align}
&(\det Rf_{*}L^{\otimes p})^{\otimes p^{n(n+2)+1}}\notag\\
 &\psi^{p}((\det Rf_{*}L^{\otimes p})^{\otimes p^{n(n+2)}})\cong\det Rf_{*}(\widetilde{\theta^{p}(\Omega_{f})^{-1}}\otimes L^{\otimes p^{2}})\notag\\
& \cong \det Rf_{*}( (-1)^{(n+1)}(\sum_{i=0}^{n+1}\binom{n+2}{i}(\theta^{p}(\Omega_{f}))^{\otimes (n+1-i)}(-p^{n})^{i})\otimes L^{\otimes p^{2}})\notag\\
&\cong\otimes_{i=0}^{n+1}(\det Rf_{*}((-1)^{(n+1+i)}\theta^{p}(\Omega_{f})^{\otimes (n+1-i)}\otimes L^{\otimes p^{2}}))^{\binom{n+2}{i}(p^{n})^{i}}\notag\\
&\cong \lambda_{n+1}^{\binom{n+2}{n+1}(p^{n})^{n+1}}\otimes\lambda_{n}^{\binom{n+2}{n}(p^{n})^{n}} \otimes\cdots\otimes\lambda_{0}\notag
\end{align}
Our theorem \ref{mrs} guarantees that $\lambda_{i}=\det Rf_{*}((-1)^{(n+1+i)}\theta^{p}(\Omega_{f})^{\otimes (n+1-i)}\otimes L^{\otimes p^{2}})$ exists functorially.

\end{proof}

\section*{Acknowledgements}
I want to thank Prof. Damian R\"{o}ssler. It is he who introduced the topic to me. I am deeply indebted to his suggestion and help. Also, I am grateful to Prof. Lin Weng who is generous to share his ideas on my questions. 

Adress of author:
\\Quan Xu
\\ Jingzhai 204, Yau mathematical Sciences center, Tsinghua University, 
\\postcode: 100084, Haidian District, Beijing China.
\\Email to: qxu@math.tsinghua.edu.cn
\\Fax:+86-10-62789445

\end{document}